\documentclass[a4paper,10pt]{article}

\usepackage{latexsym}
\usepackage{amsmath,amsthm}
\usepackage{amsfonts}

\usepackage{amssymb}

\usepackage{graphicx}
\usepackage[round]{natbib}
\usepackage{hyperref}
\hypersetup{colorlinks=true, citecolor=blue}
\usepackage{mathptmx}      

\usepackage{bbm}

\usepackage[T1]{fontenc}
\usepackage{stmaryrd}
\usepackage{esint}

\newtheorem{theorem}{Theorem}[section]

\newtheorem{lemma}[theorem]{Lemma}

\theoremstyle{definition}

\newtheorem{remark}{Remark}

\newcommand{\Card}{\operatorname{Card}}

 \usepackage[modulo,right]{lineno}

\begin{document}

\title{Stability of synchronization under stochastic perturbations in leaky integrate and fire neural networks of finite size}

\author{Pierre Guiraud$^1$ \and Etienne Tanr\'e$^2$}
\date{\today}

\footnotetext[1]{Centro de Investigaci\'on y Modelamiento de Fen\'omenos Aleatorios -- Valpara\'iso, Facultad de Ingenier\'ia, Universidad de Valpara\'{\i}so -- Valpara\'{\i}so, Chile \texttt{pierre.guiraud@uv.cl}} 
\footnotetext[2]{Universit\'e C\^ote d'Azur, Inria, France.
             2004 route des Lucioles, BP93, 06902 Sophia-Antipolis Cedex, France 
             \texttt{Etienne.Tanre@inria.fr}    }

\maketitle

\begin{abstract}
	We study the synchronization of fully-connected and totally excitatory integrate and fire neural networks in presence of Gaussian white noises. 
	Using a large deviation principle, we prove the stability of the synchronized state under stochastic perturbations.  
	Then, we give a lower bound on the probability of synchronization for networks which are not initially synchronized. This bound shows the robustness of the emergence of synchronization in presence of small stochastic perturbations.
\end{abstract}


\section{Introduction}

Neurons are the cells of the nervous system which are able to generate 
transmissible electrical signals called action potentials, or spikes for short, 
and to encode information in the sequences of time intervals separating them 
(inter-spikes intervals (ISI)). 
The information encoded in spikes sequences is 
transmitted and processed in the network formed by the neurons by means of 
synaptic interactions triggered by the spikes. 
These interactions may facilitate 
(excitatory interactions) or prevent (inhibitory interactions) the emission of
spikes in postsynaptic neurons. 
Among the collective behaviors emerging 
from this coupling, the synchronization is one of those that has attracted much 
attention in experimental and theoretical works:
a network is fully synchronized when all its neurons spike simultaneously.

Mathematical studies of synchronization have been carried out first in deterministic neural models. 
For instance, in \cite{mirollo1990synchronization}, the authors 
prove that in fully-connected and totally excitatory leaky integrate and fire neural network models the synchronization 
occurs for almost all initial state. When weak interactions are assumed, a large class of models can be reduced to canonical 
systems of phase coupled oscillators \cite{hoppensteadt1997weakly, izhikevich1999class}. This formalism allows to study 
the existence and the stability of synchronized solutions in weakly coupled general networks \cite{izhikevich1999weakly}, 
including networks containing inhibitory synapses and considering synaptic delays \cite{van1996partial,van1994inhibition}. 
Synchronized solutions and more generally phase-locked solutions in integrate and fire networks with delays have also been 
analyzed in the case of strong or arbitrary intensity coupling, see for instance \cite{bressloff1999travelling, bressloff2000dynamics}.

Successive \textit{in vivo} recordings of the membrane potential of a neuron show the existence of a large variability in its ISI. 
Although this variability could be the result of the complex encoding of information
in the neural network, it nevertheless suggests the presence of noise in the neural activity. As a matter of fact, the sources of noise are multiple. 
We can mention thermal noise, channel noise, synaptic noise, and the noise resulting of a background neural activity. 
Noise is commonly modeled
by a Gaussian stochastic process:
an additive white noise is considered in
the dynamics of the membrane potential. 
Gaussian noise models a weak interaction between the neurons of a sub-network 
and those of the 
sequel of the network. This can be derived rigorously by assuming balanced excitation and inhibition outside of the sub-network, and a Poisson statistics for the times of the presynaptic spikes \cite{gerstner2002spiking}.

The mathematical analysis and the problem of efficient numerical simulations of neural networks with Gaussian noise can be tackled when considering simplified 
models such as (leaky) integrate and fire models. 
To this class of models it is possible to associate a discrete time Markov chain containing all the informations to study the 
spiking times \cite{touboul_faugeras}. 
This approach allows to deduce qualitative properties of the network and also to build efficient algorithm to simulate exactly the ISI.
Nevertheless, the dimension of the Markov chain depends on the characteristics of the network and 
in most cases a good knowledge of the transition distribution of the chain is needed.

For instance, for Perfect Integrate and Fire (PIF) or Leaky Integrate and Fire (LIF) network models
with inhibitory interactions, the {\it countdown sequence}, which after any spiking time  
gives for each neuron the remaining time until it emits a spike supposing no interaction meanwhile,
is a Markov chain. 
For networks with excitatory interactions, to keep the Markov property of the countdown process
it is necessary to change the interaction rules \cite{turova1994synchronization} (taking into account the last presynaptic spike 
and forgetting the previous ones). Instead of changing the interaction rules, another way to have the Markov property is to consider a  chain of higher dimension composed by the pair of the countdown process and the membrane potential of each neuron at the spiking times \cite{touboul_faugeras}.
Unfortunately, there is no explicit expressions for the transition probabilities of this chain, even  for simple models as LIF networks. 
The inefficiency of the Markov chain approach makes difficult the study and simulation of neural networks with excitatory synapses. 

The mean-field limit of such networks when the number of neurons goes to infinity has also been studied \cite{DIRT,DIRT2}.
It exhibits another kind of synchronization, defined by the fact that the number of neurons which spike simultaneously is a macroscopic 
proportion of the full network (i.e \(\eta N\), where 
\(N\) is the number of neurons and \(0<\eta \leq 1\)).
In this context,  the size of the interactions has to decrease  with the number of neurons: they are of order \(1/N\).

In the present paper, we study the complete synchronization for finite networks, which can be considered as  noisy versions of the model of \cite{mirollo1990synchronization}, that is, fully-connected and totally excitatory integrate and fire neural network with noise. 
Using a large deviation principle, we show that, 
if the network is synchronized, it remains synchronized with large probability, provided the random perturbations are small enough.
In a second time, we study the emergence of synchronization in a network which is not initially synchronized.
The proof developed in \cite{mirollo1990synchronization} cannot be easily adapted, since it relies on the preservation of an order  which is destroyed in presence of noise.   
Therefore, to obtain an estimation of the probability that the synchronization occurs before the $n$-th firing of the
network, we propose an alternative strategy to those of \cite{mirollo1990synchronization, turova1994synchronization},  
which use some arguments developed for deterministic networks in \cite{Catsigeras_Guiraud}. 
This estimation is obtained supposing a specific relation 
between intensity of interaction and the number of neurons.
These results show that both the emergence of synchronization 
proved in a deterministic setting 
and the synchronized state itself are robust under small stochastic 
perturbations.

The paper is organized as follows. In Section~\ref{MODELTH}, we introduce the integrate and fire neural network with Gaussian noise.
In Section~\ref{sec:mainresult}, we give our main results on synchronization.
In Section~\ref{sec:numeric}, we present numerical simulations. All the proofs 
are contained in Section~\ref{PROOFS} and in 
the Appendix which precises 
the large deviation principle for  Ornstein-Uhlenbeck process.

\section{Model}\label{MODELTH}
We consider a leaky integrate and fire neural network with Brownian noise. 
We suppose the network contains \(N\) neurons which are labelled by an index in the set $I:=\llbracket 1,N\rrbracket$. 
For each $i\in I$, the quantity \(V^{i,\varepsilon}(t)\) represents the membrane 
potential of the neuron \(i\) at time $t\geq 0$, in presence of a noise whose intensity
is parametrized by $\varepsilon\geq 0$. 

As usual in integrate and fire neural networks, the dynamics of the model considered here has two regimes: 
a sub-threshold regime and a firing regime. We say that the network is in the sub-threshold regime if the 
potential of all the neurons is smaller than a fixed threshold potential $\theta$. 
When the potential of one neuron in the network reaches the threshold, it emits a spike and we say that the network 
is in the firing regime. Then, the potential of the spiking neurons is reset to a potential $V_r$ (smaller than \(\theta\)) and the neurons that do not spike, but are connected with at least one spiking neuron, suffer a jump in their potential depending on the synaptic weights. 

Now we detail the two regimes of the network. For all $n\geq 0$, we denote by $\tau_n^{\varepsilon}$ the $n$-th instant when the network is 
in the firing regime. More precisely, by convention $\tau_0^{\varepsilon}=0$ and for all $n\geq 0$ the instant $\tau_{n+1}^{\varepsilon}$ is the first instant 
after $\tau_{n}^{\varepsilon}$ when a spike is emitted in the network.
We suppose that the sequence $(\tau_n^\varepsilon)_{n\in\mathbb{N}}$ obeys the following induction:
\begin{equation}\label{eq:definitiontaun}
\tau^{\varepsilon}_{n+1}:=\min_{i\in I}\inf\left\{t>\tau_n^{\varepsilon}\ :\  \tilde{V}^{i,\varepsilon}_n(t)\geq\theta\right\}\quad\forall n\geq 0,
\end{equation}
where, $\tilde{V}^{i,\varepsilon}_n$ is the solution of the stochastic
differential equation 
\begin{equation}\label{eq:evolutiondevtilde}
\tilde{V}^{i,\varepsilon}_n(t)=V^{i,\varepsilon}
(\tau_n^\varepsilon) - \gamma\int_{\tau_n^\varepsilon}^{t}\left(\tilde{V}^{i,\varepsilon}_n(s)-\beta\right)ds+\sqrt{\varepsilon}\, (W_{t}^{i}-W_{\tau_n^{\varepsilon}}^{i}) \quad\forall t\geq\tau_n^{\varepsilon},
\end{equation}
where $\gamma>0$, $\beta>\theta$ and ${(W_t^1)}_{t\geq 0}, {(W_t^2)}_{t\geq 0},\dots {(W_t^N)}_{t\geq 0}$ are independent Brownian motions. The constant $\gamma$ is related to the resistance $R$ 
and the capacity $C$ of the neural membrane and is equal to $1/RC$. The parameter $\beta$ takes into account
a constant external current $I_{ext}$ and satisfies $\beta=RI_{ext}$. Note that for $n\geq 1$, the quantity $V^{i,\varepsilon}(\tau_n^\varepsilon)$ is the potential of the neuron $i$ just after the $n$-th firing regime (we will describe it soon). At time $\tau_0^\varepsilon=0$, we suppose that the potentials $V^{i,\varepsilon}(0)$ are included in $[\alpha,\theta)$, where $\alpha<0$.

It follows that for each \(n\in\mathbb{N}\) and \(i\in I\), the process \((\tilde{V}^{i,\varepsilon}_n(t),t\geq \tau_n^\varepsilon)\) is an Ornstein-Uhlenbeck process. So, it can also be written as 
\begin{equation}\label{SOLSUBTHPOT}
\tilde{V}^{i,\varepsilon}_n(t)=\phi_{t-\tau_n^{\varepsilon}}({V}^{i,\varepsilon}(\tau_n^{\varepsilon}))+
\sqrt{\varepsilon}e^{-\gamma t}\int_{\tau_n^{\varepsilon}}^{t}e^{\gamma s}\, dW_{s}^{i} \quad\forall t\geq\tau_n^{\varepsilon},
\end{equation} 
where
\[
\phi_{t}(x)=(x-\beta)e^{-\gamma t} +\beta.
\]
In the sub-threshold regime, that is for $t\in(\tau_{n}^{\varepsilon},\tau_{n+1}^{\varepsilon})$ with $n\geq 0$, we suppose that the membrane potential of each neuron satisfies 
\begin{equation}\label{POTSUBTH}
V^{i,\varepsilon}(t)=\tilde{V}^{i,\varepsilon}_n(t). 
\end{equation}

Now we describe the firing regime. 
We denote by $V^{i,\varepsilon}(\tau_n^{\varepsilon}-)$   the value of the potential of neuron $i$, in the sub-threshold regime just before the instant $\tau_n^{\varepsilon}$. It is a shorthand notation for:
\[
V^{i,\varepsilon}(\tau_n^{\varepsilon}-):=\lim_{t\nearrow\tau_n^{\varepsilon}}V^{i,\varepsilon}(t)
=\phi_{\tau_n^{\varepsilon}-\tau_{n-1}^{\varepsilon}}(V^{i,\varepsilon}(\tau_{n-1}^{\varepsilon}))+
\sqrt{\varepsilon}e^{-\gamma\tau_n^{\varepsilon}}\int_{\tau_{n-1}^{\varepsilon}}^{\tau_n^{\varepsilon}}e^{\gamma s}\, dW_{s}^{i}.
\]
We now define step by step the set \(J(n)\)  of neurons which spike at time \(\tau_n^{\varepsilon}\).
The first component of \(J(n)\) is
\[
J^0(n)=\{i\in I,\quad V^{i,\varepsilon}(\tau_n^{\varepsilon}-)\geq \theta\}.
\]
The set \(J^0(n)\) corresponds to the neurons which spike spontaneously (i.e. without any interaction with the other neurons 
in the network). In practice, in our model, the set \(J^0(n)\) is almost surely a singleton. We then define
\[
J^{1}(n)=\{i\in I\setminus J^{0}(n),\quad V^{i,\varepsilon}(\tau_n^{\varepsilon}-)+\sum_{j\in J^{0}(n)}H_{ji}\geq \theta\},
\]
where the constants $H_{ji}$ are the synaptic weights. They represent the effect on the neuron $i$ of a spike emitted by the neuron $j$. 
We assume the network fully connected and 
purely excitatory
\begin{equation}
m:= \min_{i,j\in I}H_{ji} > 0.
\end{equation}
Note that a neuron $i$  of \( J^1(n)\) has a membrane potential smaller than \(\theta\) at time \(\tau_n^{\varepsilon}-\) but
larger than \(\theta\) after receiving at time \(\tau_n^{\varepsilon}\) the kicks \(H_{ji}\) of the spiking neurons \(j\) of \(J^0(n)\). The set $J^1(n)$
is therefore the set of the neurons which spike at time $\tau_n^\varepsilon$  by interaction with the neurons of \(J^0(n)\).
In the same way, 
we define by induction the sets
\[
J^{p+1}(n)=\{i\in I\setminus \cup_{0\leq q\leq p}J^{q}(n),\quad V^{i,\varepsilon}(\tau_n^{\varepsilon}-)+\sum_{j\in\cup_{0\leq q\leq p}J^{q}(n)}H_{ji}\geq \theta\}.
\]
Thus,
\[
J(n):=\cup_{p\in\mathbb{N}}J^{p}(n)
\]
is the set of the labels of the neurons which emit a spike at time \(\tau_n^\varepsilon\). Note that this union is finite, since 
our definition implies that \(J^{p+1}(n) \subset (\cup_{0\leq q\leq p}J^{q}(n))^\mathrm{C}\). Therefore, a neuron can not spikes twice (or more)
at time \(\tau_n^{\varepsilon}\). Another way to obtain this property could be to introduce a refractory period in the previous model,
but it would add technical difficulties. Nevertheless, the model considered in this paper can be interpreted as 
the limit of the similar model with a refractory period, when the refractory period goes to zero.

Now we can define the state of the network at time \(\tau_n\), at the end of the
firing regime, by 
\begin{equation}\label{FIRE!}
V^{i,\varepsilon}(\tau_n^\varepsilon)=
\begin{cases}
\begin{array}{ll}
V_r & \mbox{ if } i\in J(n)\\
V^{i,\varepsilon}(\tau_n^{\varepsilon}-) + \sum\limits_{j\in J(n)} H_{ji} & \mbox{ if } i\notin J(n).
\end{array}
\end{cases}
\end{equation}
The definition of \(J(n)\) ensures that  
\(V^{i,\varepsilon}(\tau_n^{\varepsilon}) <\theta\) forall \(i\in I\). This ends the $n$-th firing regime and the network starts again to evolve according to the  equation \eqref{POTSUBTH} of the sub-threshold regime.

\section{Results}\label{sec:mainresult}
Recall that the network is synchronized at time \(t\) if all the neurons spike
simultaneously at time \(t\). We have already introduced the sequence of spiking times
\(\tau_1^\varepsilon, \cdots, \tau_n^\epsilon, \cdots\) at which at least one neuron emits a spike.
Here, we are interested in events of the form 
\begin{align}
\label{eq:defevenements}
S_k^\varepsilon & := \{\mbox{ the network is synchronized at time } \tau_k^\varepsilon \}\\
\label{eq:defsynchroavantn}
BS_n^\varepsilon&:= \{\mbox{ the network has been synchronized at least once before time }  \tau_n^\varepsilon\}.
\end{align} 
Obviously, we have
\begin{align*}
BS_n^\varepsilon = \cup_{k=1}^n S_k^\varepsilon\\
S_k^\varepsilon = \{J(k) = I\}.
\end{align*}
In a deterministic framework, once the network is synchronized it remains synchronized for ever. But when stochastic perturbations
are considered the potentials of the neurons do not remain equal after synchronization. Therefore, it is 
not guarantied that the network get synchronized again at the next spiking time. Nevertheless, an estimation of the probability of this event is provided by the following theorem:
\begin{theorem}\label{THSTABSYNCHRO}
	There exists \(\varepsilon_0 > 0\), such that for every \(\varepsilon \in(0, \varepsilon_0]\) the probability that the network 
	stays synchronized at time \(\tau_{n+1}^\varepsilon\) given that  for some integer $n$ it is synchronized at time \(\tau_n^\varepsilon\) satisfies, 
	\[
	\mathbb{P}\left(S_{n+1}^\varepsilon\middle|S_{n}^\varepsilon\right)\geq
	\left(1-\exp\left(-\gamma\frac{m^2}{4\varepsilon}\right)\right)^{N}.
	\]
\end{theorem}
This result essentially shows that the synchronization  is not destabilized by small stochastic perturbations. 
More precisely, if the intensity $\varepsilon$ of the noise goes to $0$, then the probability that the network stays
synchronized goes to \(1\), and for sufficiently small intensity of the noise the network stays synchronized with high probability.
The proof of Theorem~\ref{THSTABSYNCHRO} is given in Section \ref{PROOFSTABSYNCHRO}. 
The main idea is to use a large deviations principle to obtain a lower bound on the probability that the potential of each neuron 
remains close to the deterministic synchronized solution until some neurons reach the threshold potential. 
In this case, the potential of the neurons are sufficiently close to each other at the firing time,
and the interactions between neurons induce the synchronization.

Now we study the probability that the networks get synchronized.

\begin{theorem}\label{th:mainresultsynch} 
	Assume \((V^{i,\varepsilon}(t), i\in\llbracket 1,N\rrbracket,  t\geq 0)\) evolves according to \eqref{POTSUBTH} and \eqref{FIRE!} and that the number \(N\) of neurons 
	in the network satisfies 
	\begin{equation}\label{NUMNEURON}
	N \geq \dfrac{\theta-\alpha}{m}\left(\dfrac{\theta-\alpha}{m}+2\right).
	\end{equation}
	Then, for all \(n\in [\frac{\theta-\alpha}{m},\frac{mN}{\theta-\alpha}]\), the probability
	that the network synchronizes during the first \(n\) spiking times satisfies 
	\begin{multline}\label{BORNESYNCHRO}
	\mathbb{P}\left(BS_n^\varepsilon\right) \geq \left(1 - N\Phi\left(-m\sqrt{\dfrac{2\gamma}{\varepsilon}} \inf\left\{n-\frac{\theta-\alpha}{m} ,\frac{\beta-\theta}{m}\right\}\right) \right)_+\\
	\times\left(1  - \Phi\left(-m\sqrt{\dfrac{2\gamma}{\varepsilon}}\left(\frac{N}{n} - \frac{\theta - \alpha}{m}\right)\right) - n\Phi\left(-m\sqrt{\dfrac{2\gamma}{\varepsilon}}\right)\right)^N_+,
	\end{multline}
	where \(\Phi\) denotes the cumulative distribution function of a standard Gaussian random variable and \((x)_+:= \sup(x,0)\) denotes the positive part of a real number \(x\).
\end{theorem}
Note that the condition \eqref{NUMNEURON} ensures that the integer $n_0=\left\lceil \tfrac{\theta-\alpha}{m}\right\rceil+1$
always belongs to the interval $[\frac{\theta-\alpha}{m},\frac{mN}{\theta-\alpha}]$. 
The sequence of events \((BS_n^\varepsilon)_n\) is obviously increasing. So, we deduce from Theorem~\ref{th:mainresultsynch} that for every $n\geq n_0$ (without any upper bound constraint on \(n\))
\begin{align}
\nonumber\mathbb{P}\left(BS_n^\varepsilon\right)
\geq& \mathbb{P}\left(BS_{n_0}^\varepsilon\right)\\
\nonumber \geq&\left(1 - N\Phi\left(-p_1^\varepsilon \inf\{1+ \lceil p_2\rceil -p_2,p_3\}\right)\right)_+\\
\nonumber&\times\left(1  - \Phi\left(-p_1^\varepsilon \frac{N-p_2\lceil p_2+1\rceil}{\lceil p_2+1\rceil}\right) - \lceil p_2+1\rceil\Phi\left(-p_1^\varepsilon \right)\right)_+^N\\
\geq&\left(1 - N\Phi(-p_1^\varepsilon \inf\{1,p_3\})\right)_+(1  - \Phi(-\tfrac{N p_1^\varepsilon}{\lceil p_2+1\rceil} + p_1^\varepsilon p_2) - (p_2+2)\Phi\left(-p_1^\varepsilon \right))_+^N,
\label{BOUNDPAM}
\end{align}
where \(p_1^\varepsilon\), $p_2$ and $p_3$ are the dimensionless parameters defined by 
\[
p_1^\varepsilon:=m\sqrt{\dfrac{2\gamma}{\varepsilon}}\qquad p_2:=\frac{\theta-\alpha}{m}\quad\text{and}\quad p_3:=\frac{\beta-\theta}{m}.
\] 
Note also that condition \eqref{NUMNEURON} can be written as $N \geq p_2(p_2+2)$. An analysis of \eqref{BOUNDPAM} shows that the 
lower bound on $\mathbb{P}(BS_n^\varepsilon)$ is an increasing function of $p_1^\varepsilon$ and a decreasing function of $p_2$.
Therefore, if $n\geq n_0$ the lower bound on
\(\mathbb{P}(BS_n^\varepsilon)\)
increases when the intensity \(\varepsilon\) of noise decreases, and goes to $1$ when  
\(\varepsilon\) 
goes to $0$. This result shows that the phenomenon of synchronization observed in deterministic networks, such as those of \cite{mirollo1990synchronization}, 
is stable under stochastic perturbations. In other words, in presence of small stochastic perturbations, the synchronization occurs in a finite time 
(smaller than $\tau_n^\varepsilon$) with a high probability which goes to 1 in the deterministic limit. 
Moreover as expected in excitatory networks, strong interaction between neurons facilitates the synchronization.
Indeed, our lower bound on the probability of synchronization is an increasing function of the interaction parameter \(m\).
Furthermore, with larger \(m\) our estimation is valid for smaller number of neurons (see \eqref{NUMNEURON}) and smaller number of spiking times (\(n\simeq \tfrac{\theta - \alpha}{m}\)).

\begin{remark}
	In the proof of Theorem~\ref{THSTABSYNCHRO}, we obtain a control on the probability that \(\Card(J^1(n+1))=N-1\) given \(J(n)=I\)
	which is not a necessary condition for synchronization. Indeed, order the membrane potential of the neurons at time \(\tau^\varepsilon_{n+1}-\)
	\[
	\theta= V^{\kappa(1),\varepsilon}(\tau_{n+1}^{\varepsilon}-) > V^{\kappa(2),\varepsilon}(\tau_{n+1}^{\varepsilon}-)  > \cdots > V^{\kappa(N),\varepsilon}(\tau_{n+1}^{\varepsilon}-) .
	\]
	A sufficient condition for synchronization is
	\[
	V^{\kappa(i),\varepsilon}(\tau_{n+1}^{\varepsilon}-) \geq \theta - (\kappa(i) - 1)m, \quad \forall i \in \llbracket 1, N\rrbracket.
	\]
	This last condition is weaker than the bound we use in the proof of 
	Theorem~\ref{THSTABSYNCHRO} but, 
	unfortunately, even  the law of \(V^{\kappa(i),\varepsilon}(\tau_{n+1}^{\varepsilon}-)\) is out of reach: it is related to the order statistic
	of a family of \(N\) Ornstein-Uhlenbeck processes at the first time at which one of them hits \(\theta\). This law is unknown.
	\cite{sacerdote_zucca2013} give first results in this direction in a very simplified setting with \(N=2\) neurons and \(\gamma = 0\) (Perfect Integrate
	and Fire model). See also a recent review for the LIF model with \(N=1\) in \cite{sacerdote_giraudo2013}.  
	Similar results in our context should permit to accurately  evaluate the probability of synchronization.
\end{remark}

\section{Numerical simulations}\label{sec:numeric}

In this section, we illustrate our theorems with figures obtained by numerical simulations of the model.

\begin{figure}[!ht]
	\begin{center}
		\resizebox{11cm}{!}{\includegraphics{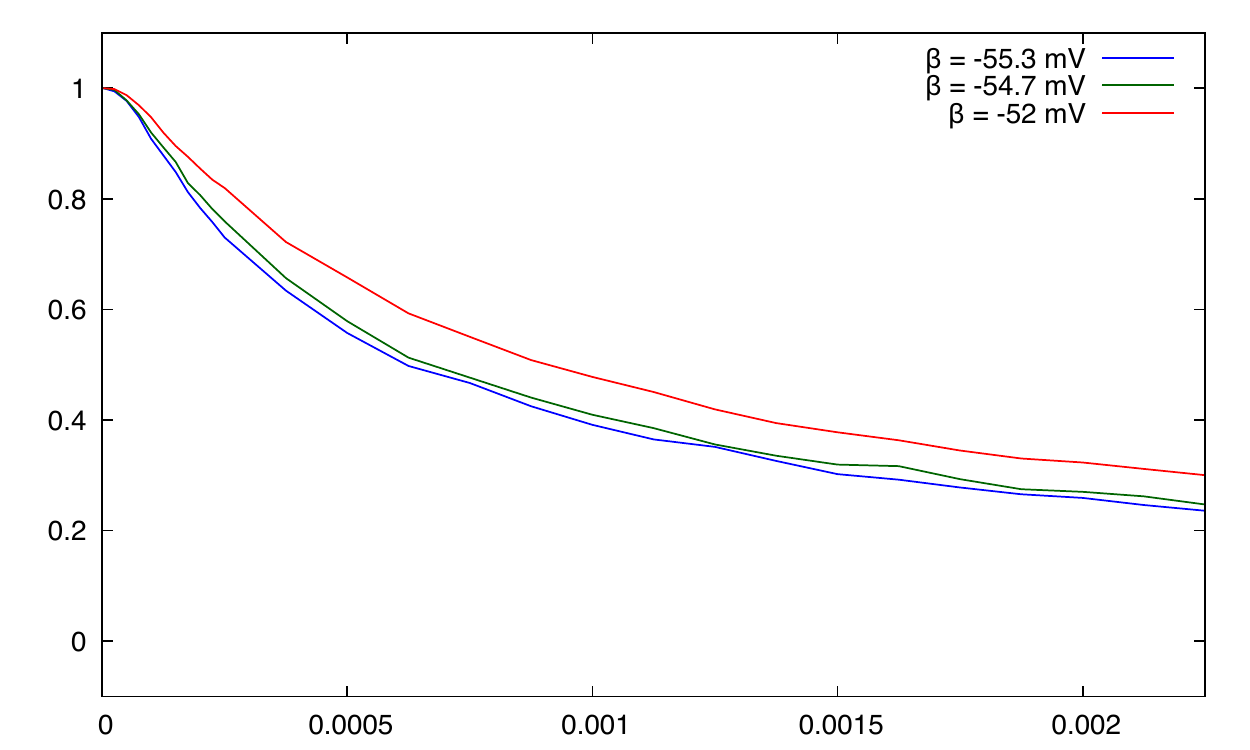}}
		\resizebox{11cm}{!}{\includegraphics{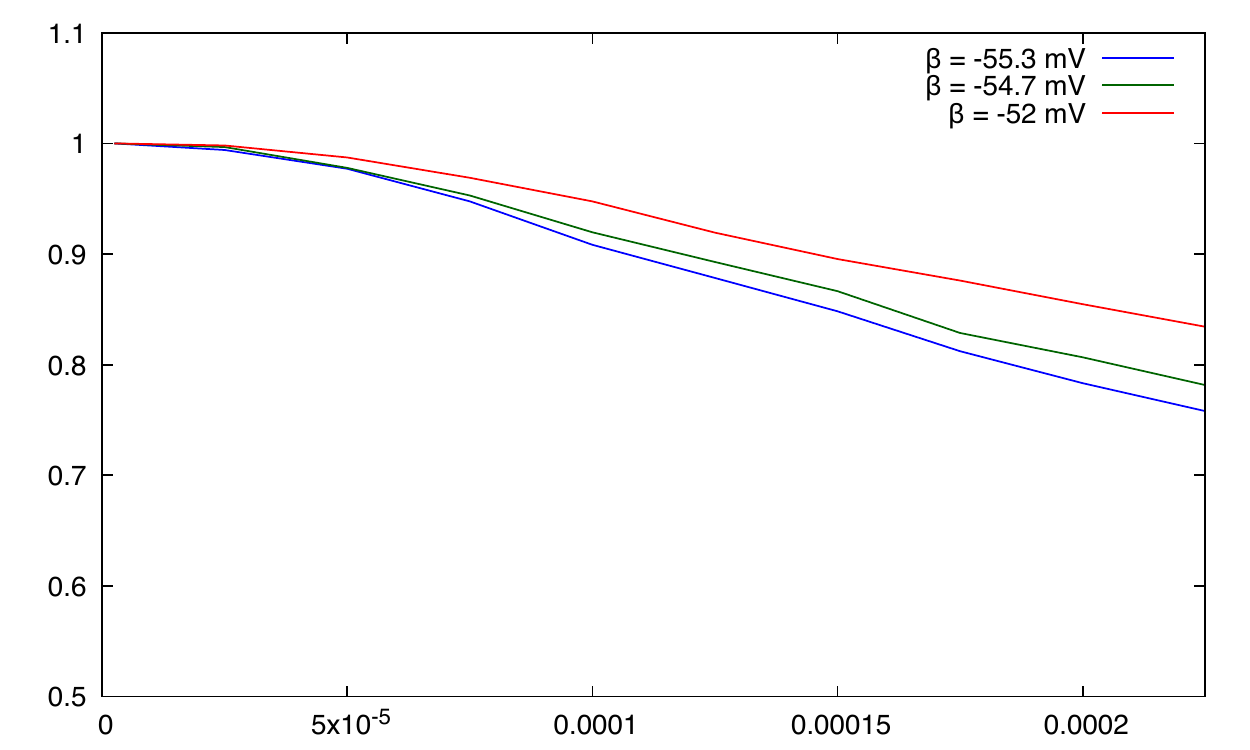}}
		\caption{$\mathbb{P}\left(S_{n+1}^\varepsilon\middle|S_{n}^\varepsilon\right)$ versus $\varepsilon$ for $N=1599$, \(\gamma = 100\) $\text{s}^{-1}$, \(\theta = -55\) mV, $V_r=-70$ mV, $m=0.75$ mV. Red plot \(\beta = -52\) mV, green plot  \(\beta=-54.7\) mV  and blue plot \(\beta = -55.3\) mV. In the upper panel $\varepsilon\in[0, 2.25\times 10^{-3}\text{V}^2/\text{s}]$ and  the lower panel is a zoom in the range $\varepsilon\in[0,2.25\times 10^{-4}\text{V}^2/\text{s}]$.}  
		\label{FIGTH1}
	\end{center}
\end{figure}

In Figure~\ref{FIGTH1}, we illustrate Theorem~\ref{THSTABSYNCHRO}: we  
plot as function of the intensity of the noise \(\varepsilon\) 
the probability of the network to be  synchronized at the $(n+1)$-th firing time, given that it was synchronized at the $n$-th firing time. For this figure, the simulation has been done with $N=1599$ neurons,  \(\gamma = 100\) $\text{s}^{-1}$, \(\theta = -55\) mV, $V_r=-70$ mV, $m=0.75$ mV,
\(H_{ji} = m\) for all \(i, j\) in \(\llbracket 1,n\rrbracket\), and three values of 
$\beta$: \(\beta = -52\) mV (red), \(\beta=-54.7\) mV (green) and \(\beta = -55.3\) mV (blue). 
The two first values of $\beta$ satisfy our hypothesis  \(\beta>\theta\), but not the last one.

As predicted by our theorem, for $\beta>\theta$ the plots show that once the network is synchronized, it stays synchronized with very large  
probability, provided the noise is sufficiently small. In other words, the synchronized state is stable under small stochastic perturbations.

When $\beta<\theta$, no spike are produced in the deterministic case ($\varepsilon=0$). However, in presence of noise ($\varepsilon>0$)  
spikes are emited even if $\beta<\theta$. 
We observe no significative difference  in the behavior of 
$\mathbb{P}\left(S_{n+1}^\varepsilon\middle|S_{n}^\varepsilon\right)$ as a function of $\varepsilon$ for  $\beta=-55.3\text{ mV}<\theta$ and  $\theta<\beta=-54.7\text{ mV}$.

\begin{figure}[!ht]
	\begin{center}
		\resizebox{11cm}{!}{\includegraphics{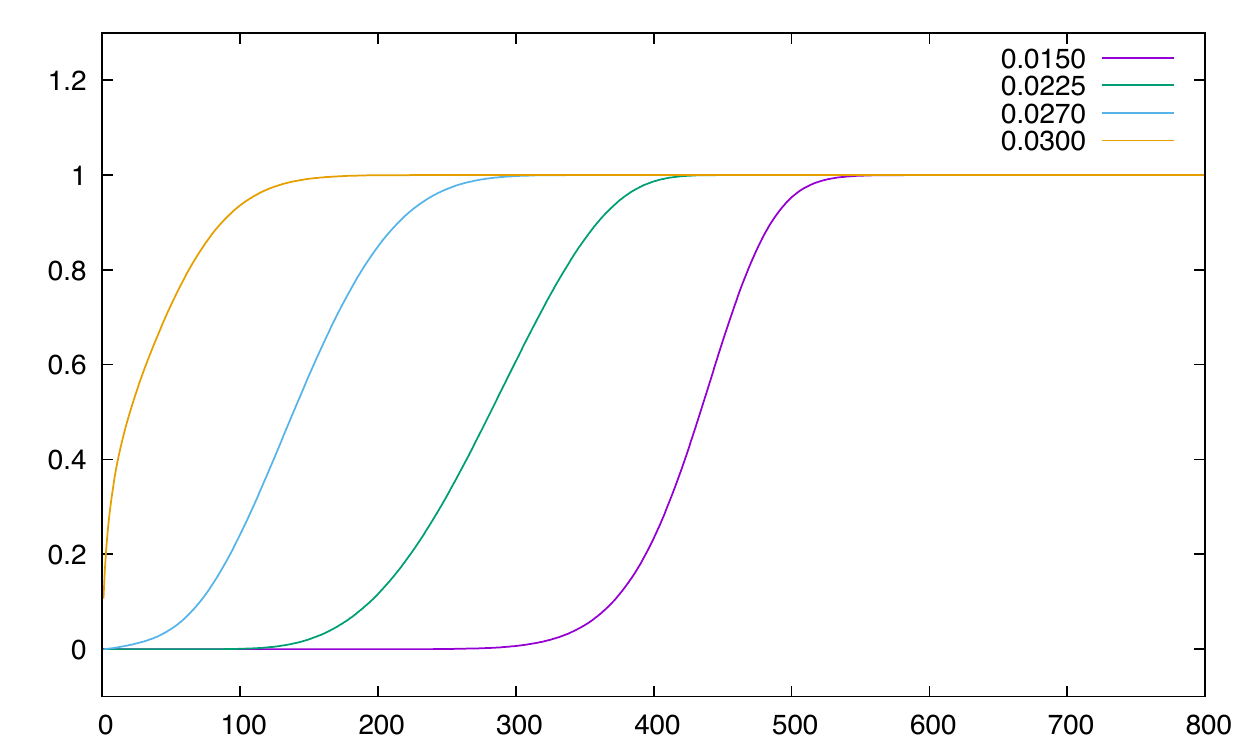}}
		\resizebox{11cm}{!}{\includegraphics{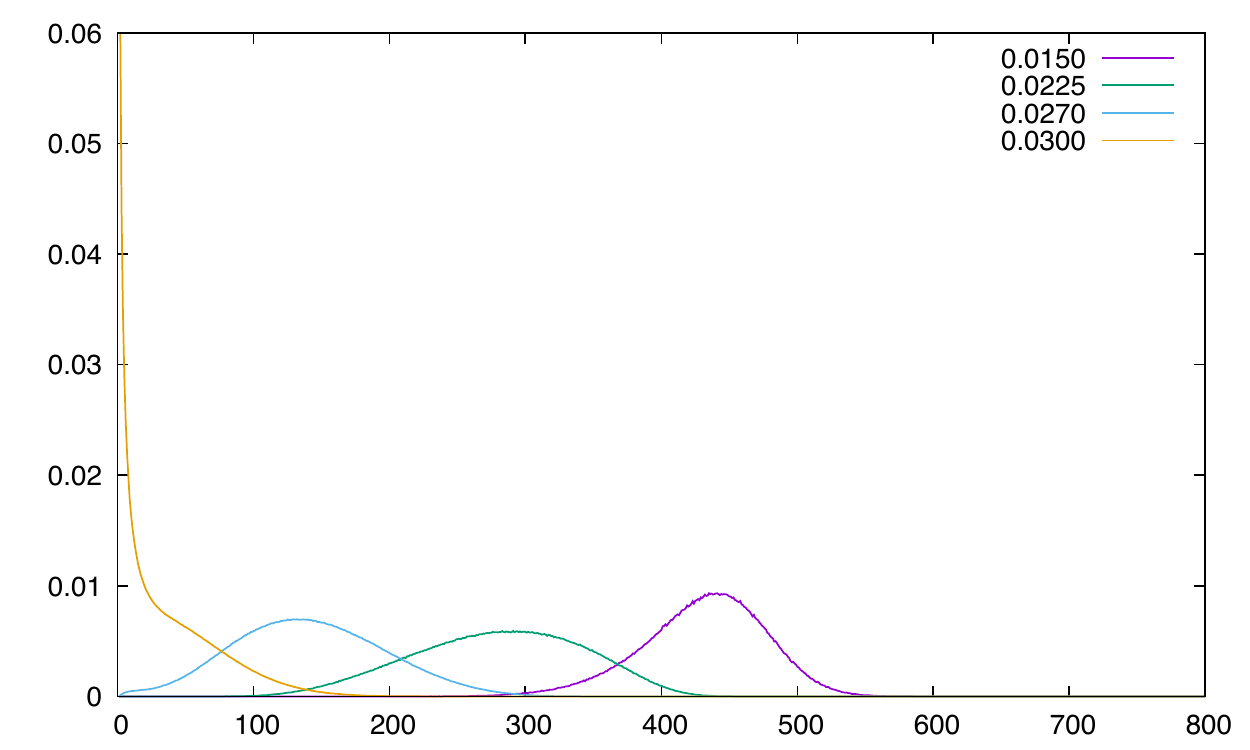}}
		\caption{$\mathbb{P}\left(BS_{n}^\varepsilon\right)$ versus $n$ (upper panel) and $\mathbb{P}\left(BS_{n}^\varepsilon \setminus BS^\varepsilon_{n-1}\right)$ probability of first synchronization at the $n$-th firing time versus $n$ (lower pannel). The values of
			the parameters are:  $N=1599$, $\gamma=100$ $\text{s}^{-1}$, $\theta=-55$ mV, $V_r=-70$ mV , $\beta=-52$ mV  
			and $\varepsilon=0,225\times 10^{-4}\text{V}^2/\text{s}$. Magenta $m=0.0150$ mV, green $m=0.0225$ mV, blue  $m=0.0270$ mV and yellow  $m=0.0300$ mV.
			The initial distribution of the membrane potentials are i.i.d. with uniform
			law on \([-100\text{ mV},-55\text{ mV}]\).}  
		\label{FIGTH2}
	\end{center}
\end{figure}

Figure~\ref{FIGTH2} shows in the upper panel the probability $\mathbb{P}\left(BS_{n}^\varepsilon\right)$ of synchronization before the $n$-th firing of the network as a function of $n$.
In the lower panel, the figure shows the probability $\mathbb{P}\left(BS_{n}^\varepsilon \setminus BS^\varepsilon_{n-1}\right)$ of first synchronization at the $n$-th firing time versus $n$.
The plots are done for several values of the intensity of the interactions. For the deterministic model ($\varepsilon=0$),  synchronization occurs in a finite time \cite{mirollo1990synchronization, Catsigeras_Guiraud}. In Figure \ref{FIGTH2}, we see that even in presence of noise the synchronization occurs in a short time with a high probability. 
We also observe the same monotony properties 
as the upper bound \eqref{BOUNDPAM}: the higher is the intensity of the interactions, the higher is the probability that the network synchronizes fast.
The initial condition of the membrane potential is
i.i.d. with a uniform distribution on \([-100\text{ mV}, -55\text{ mV}]\).

\section{Proofs}\label{PROOFS}
\subsection{Proof of the Result of Synchronization (Th.~\ref{th:mainresultsynch})}
The proof consists in estimating  
\begin{align}
\mathbb{P}\left(\sup_{\ell \in\llbracket 1,n\rrbracket} \Card(J(\ell)) = N\middle|
\sup_{\ell \in\llbracket 1,n\rrbracket} \Card(J(\ell)) \geq A\right)\label{eq:synchrocond}
\\ 
\quad\text{and}\quad
\mathbb{P}\left(\sup_{\ell \in\llbracket 1,n\rrbracket} \Card(J(\ell)) \geq A\right),\label{eq:cardjsupa}
\end{align}
in order to obtain
\[
\mathbb{P}(BS_n^\varepsilon) = \mathbb{P}\left(\sup_{\ell \in\llbracket 1,n\rrbracket} \Card(J(\ell)) = N\right)
\]
using Bayes formula. A lower bound of \eqref{eq:cardjsupa} is obtained in Lemma~\ref{lemma minorantchaussette} for \(A=\frac{N}{n}\). The control of the quantity \eqref{eq:synchrocond}  is given in Lemma~\ref{lemmaprobasynchrodechargebcp}.

We first prove a technical lemma for the
explicit expression of the potential \(V^{i,\varepsilon}\)
before the first spiking time of neuron \(i\).

\newpage

\begin{lemma}\label{POTTAUN}
	For all $n\geq 1$ and $i\in I$ we have

	\begin{equation}
	\begin{split}
	\mathbbm{1}_{\left\{i\notin\cup_{\ell=1}^n J(\ell)\right\}}
	V^{i,\varepsilon}(\tau^{\varepsilon}_n)=&\mathbbm{1}_{\left\{i\notin\cup_{\ell=1}^n J(\ell)\right\}}
	\left(\phi_{\tau^{\varepsilon}_n}(V^{i,\varepsilon}(0))+\sum_{\ell=1}^n
	e^{-\gamma(\tau^{\varepsilon}_n-\tau_\ell^\varepsilon)}\sum_{j\in J(\ell)}H_{ji}\right.\\
	&\left.+\sqrt{\varepsilon}e^{-\gamma \tau^{\varepsilon}_n} \int_0^{\tau^{\varepsilon}_n} e^{\gamma s} dW^i_s\right),
	\end{split}
	\end{equation}
	where \(\mathbbm{1}_{S}\) denotes the indicator function of a set \(S\), that is \(\mathbbm{1}_{S}(\omega)=1\) for \(\omega\in S\) and 
	\(\mathbbm{1}_{S}(\omega)=0\) for \(\omega\notin S\).
\end{lemma}
\begin{proof}
	By formula (\ref{FIRE!}) 
	\[
	\mathbbm{1}_{\left\{i\notin J(1)\right\}}V^{i,\varepsilon}(\tau^{\varepsilon}_1)=\mathbbm{1}_{\left\{i\notin J(1)\right\}}
	\left(\phi_{\tau^{\varepsilon}_1}(V^{i,\varepsilon}(0))+\sum_{j\in J(1)}H_{ji}+\sqrt{\varepsilon}e^{-\gamma \tau^{\varepsilon}_1} \int_0^{\tau^{\varepsilon}_1} e^{\gamma s} dW^i_s\right)
	\]
	and the lemma is true for $n=1$. From the same formula (\ref{FIRE!}) at time $\tau_{n+1}^{\varepsilon}$ we have
	\begin{equation}\label{VTAUNP1}
	\begin{split}
	\mathbbm{1}_{\left\{i\notin\cup_{\ell=1}^{n+1} J(\ell)\right\}}V^{i,\varepsilon}(\tau^{\varepsilon}_{n+1})=&\mathbbm{1}_{\left\{i\notin\cup_{\ell=1}^{n+1} J(\ell)\right\}}
	\left(\phi_{\tau^{\varepsilon}_{n+1}-\tau^{\varepsilon}_{n}}(V^{i,\varepsilon}(\tau^{\varepsilon}_{n}))+\sum_{j\in J(n+1)}H_{ji}\right.\\
	&\left.+\sqrt{\varepsilon}e^{-\gamma \tau^{\varepsilon}_{n+1}} \int_{\tau^{\varepsilon}_{n}}^{\tau^{\varepsilon}_{n+1}} e^{\gamma s} dW^i_s\right).
	\end{split}
	\end{equation}
	Supposing the lemma is true at rank $n$ and using $\phi_t(x+y)=\phi_t(x)+ye^{-\gamma t}$ we obtain
	\begin{equation}\label{PHIREC}
	\begin{split}
	\phi_{\tau^{\varepsilon}_{n+1}-\tau^{\varepsilon}_{n}}\left(\mathbbm{1}_{\left\{i\notin\cup_{\ell=1}^{n} J(\ell)\right\}}V^{i,\varepsilon}(\tau^{\varepsilon}_{n})\right)=
	&\phi_{\tau^{\varepsilon}_{n+1}-\tau^{\varepsilon}_{n}}\left(\mathbbm{1}_{\left\{i\notin\cup_{\ell=1}^{n} J(\ell)\right\}}\phi_{\tau^{\varepsilon}_{n}}\left(V^{i,\varepsilon}(0)\right)\right)\\
	&+\left(\sum_{\ell=1}^n
	e^{-\gamma(\tau^{\varepsilon}_{n+1}-\tau_\ell^\varepsilon)}\sum_{j\in J(\ell)}H_{ji}\right.\\
	&\left.+\sqrt{\varepsilon}e^{-\gamma \tau^{\varepsilon}_{n+1}} \int_0^{\tau^{\varepsilon}_n} e^{\gamma s} dW^i_s\right)\mathbbm{1}_{\left\{i\notin\cup_{\ell=1}^{n} J(\ell)\right\}}.
	\end{split}
	\end{equation}
	Using 
	\[
	\mathbbm{1}_{\left\{i\notin\cup_{\ell=1}^{n+1} J(\ell)\right\}}=\mathbbm{1}_{\left\{i\notin\cup_{\ell=1}^{n+1} J(\ell)\right\}}\mathbbm{1}_{\left\{i\notin\cup_{\ell=1}^{n} J(\ell)\right\}}
	\]
	and
	\[
	\mathbbm{1}_{\left\{i\notin\cup_{\ell=1}^{n+1} J(\ell)\right\}}\phi_t\left(x\right)=\mathbbm{1}_{\left\{i\notin\cup_{\ell=1}^{n+1} J(\ell)\right\}}\phi_t\left(\mathbbm{1}_{\left\{i\notin\cup_{\ell=1}^{n} J(\ell)\right\}}x\right),
	\]
	from equations (\ref{VTAUNP1}) and (\ref{PHIREC}), we deduce
	\begin{equation*}
	\begin{split}
	\mathbbm{1}_{\left\{i\notin\cup_{\ell=1}^{n+1} J(\ell)\right\}}V^{i,\varepsilon}(\tau^{\varepsilon}_{n+1})
	=&\mathbbm{1}_{\left\{i\notin\cup_{\ell=1}^{n+1} J(\ell)\right\}}\left(
	\phi_{\tau^{\varepsilon}_{n+1}-\tau^{\varepsilon}_{n}}\left(\phi_{\tau^{\varepsilon}_{n}}\left(V^{i,\varepsilon}(0)\right)\right)
	+\sum_{j\in J(n+1)}H_{ji}\right.\\
	&+\sqrt{\varepsilon}e^{-\gamma \tau^{\varepsilon}_{n+1}} \int_{\tau^{\varepsilon}_{n}}^{\tau^{\varepsilon}_{n+1}} e^{\gamma s} dW^i_s
	+\sum_{\ell=1}^n e^{-\gamma(\tau^{\varepsilon}_{n+1}-\tau_\ell^\varepsilon)}\sum_{j\in J(\ell)}H_{ji} \\
	&\left.+\sqrt{\varepsilon}e^{-\gamma \tau^{\varepsilon}_{n+1}} \int_0^{\tau^{\varepsilon}_n} e^{\gamma s} dW^i_s\right),
	\end{split}
	\end{equation*}
	and it immediately follows that the property is true at rank $n+1$ 
\end{proof}

\begin{lemma}\label{lemma minorantchaussette}
	Assume the support of the initial condition
	is included in \((\alpha,\theta)\).
	Let $n\geq\frac{\theta-\alpha}{m}$ then,
	\[
	\mathbb{P}\left(\max_{\ell\in\llbracket 1,n\rrbracket} \Card(J(\ell)) >  \frac{N}{n} \right)
	\geq
	1 - N\Phi\left(-m\sqrt{\frac{2\gamma}{\varepsilon}}\inf\left\{n - \frac{\theta-\alpha}{m},\frac{\beta-\theta}{m}\right\}\right) 
	\]
\end{lemma}
\begin{proof}
	Since, 
	\[
	\left(\bigcup_{\ell=1}^nJ(\ell) = I\right)\Longrightarrow\left(\exists \ell\in\llbracket 1,n\rrbracket\quad s.t. \quad \Card(J(\ell))\geq\frac{N}{n} \right),
	\]
	we have
	\begin{equation}\label{eq chaussettes}
	\mathbb{P}\left(\max_{\ell\in\llbracket 1,n\rrbracket} \Card(J(\ell)) >  \frac{N}{n} \right) \geq
	\mathbb{P}\left( \bigcup_{\ell=1}^nJ(\ell) = I \right).
	\end{equation}
	From Lemma~\ref{POTTAUN} we deduce that for any $i\in I$
	\begin{align*}
	\mathbb{P}\left(i\notin \bigcup_{\ell=1}^nJ(\ell)\right)
	&\leq
	\mathbb{P}\left(\phi_{\tau^\varepsilon_n}(V^{i,\varepsilon}(0))+\sum_{\ell=1}^n
	e^{-\gamma(\tau^\varepsilon_n-\tau_\ell^\varepsilon)}\sum_{j\in J(\ell)}H_{ji}+\sqrt{\varepsilon}e^{-\gamma \tau^\varepsilon_n} \int_0^{\tau^\varepsilon_n} e^{\gamma s} dW^i_s<\theta\right)\\
	&\leq
	\mathbb{P}\left(\phi_{\tau^\varepsilon_n}(V^{i,\varepsilon}(0))+mne^{-\gamma\tau^\varepsilon_n}+\sqrt{\varepsilon}e^{-\gamma \tau^\varepsilon_n} \int_0^{\tau^\varepsilon_n} e^{\gamma s} dW^i_s<\theta\right)\\
	&\leq
	\mathbb{P}\left(\phi_{\tau^\varepsilon_n}(V^{i,\varepsilon}(0)+mn)+\sqrt{\varepsilon}e^{-\gamma \tau^\varepsilon_n} \int_0^{\tau^\varepsilon_n} e^{\gamma s} dW^i_s<\theta\right)\\
	&\leq
	\mathbb{P}\left(\inf\{V^{i,\varepsilon}(0)+mn,\beta\}+\sqrt{\varepsilon}e^{-\gamma \tau^\varepsilon_n} \int_0^{\tau^\varepsilon_n} e^{\gamma s} dW^i_s<\theta\right).
	\end{align*}
	Now, remind that 
	\[
	\mathcal{L}\left(\sqrt{\varepsilon}e^{-\gamma t} \int_0^{t} e^{\gamma s} dW_s\right)=\mathcal{N}\left(0,\varepsilon\frac{1-e^{-2\gamma t}}{2\gamma}\right)\qquad\forall\, t\in\mathbb{R}.
	\]
	On the other hand, by monotonicity of the cumulative distribution
	function \(\Phi\), for all \(a \geq a_0\) and \(\sigma^2 \leq \sigma_0^2\), we have
	\begin{equation}\label{eq comparaisonfnrepart}
	\forall x \leq a, \quad  \mathbb{P}\left(\mathcal{N}(a,\sigma^2)  < x\right) \leq \Phi\left(\frac{x-a_0}{\sigma_0}\right).
	\end{equation}
	Then, with \(a_0 = \inf\{\alpha+mn,\beta\}\), \(a = \inf\{V^{i,\varepsilon}(0)+mn,\beta\}\), \(x = \theta\), 
	\(\sigma_0 = \sqrt{\frac{\varepsilon}{2\gamma}}\) and \(\sigma = 
	\sqrt{\frac{1-e^{-2\gamma t}}{2\gamma}}\), 
	if $\inf\{\alpha+mn,\beta\}\geq \theta$, we have 
	\begin{align*}
	\mathbb{P}\left(i\notin \bigcup_{\ell=1}^nJ(\ell)\right)
	&\leq
	\Phi\left(\frac{\theta-\inf\{\alpha+mn,\beta\}}
	{\sqrt{\frac{\varepsilon}{2\gamma}}}\right) \\
	&= \Phi\left(-m\sqrt{\frac{2\gamma}{\varepsilon}}\inf\left\{n - \frac{\theta-\alpha}{m},\frac{\beta-\theta}{m}\right\}\right) .
	\end{align*}
	The result follows then from \eqref{eq chaussettes} and
	\[
	\mathbb{P}\left(\bigcup_{\ell=1}^nJ(\ell)\neq I\right)
	\leq\sum_{i=1}^N\mathbb{P}\left(i\notin \bigcup_{\ell=1}^nJ(\ell)\right).
	\]
\end{proof}

\begin{lemma}\label{lemmaqueuededistribution}
	For all \(n\geq 0\), \(t\geq \tau_n\), \(i\in\llbracket 1,N\rrbracket\), 
	\begin{equation}\label{eq inegalitelemme31}
	\forall C \geq  \theta - \alpha, \quad
	\mathbb{P}\left(\tilde{V}^{i,\varepsilon}_n(t) + C < \theta\right)
	\leq  \Phi\left((\theta-\alpha-C)\sqrt{\dfrac{2\gamma}{\varepsilon}}\right) + n\Phi\left(-m\sqrt{\frac{2\gamma}{\varepsilon}}\right).
	\end{equation}
\end{lemma}
\begin{proof}
	We start with \(n=0\). For all \(t\), conditionally to
	the initial value \(V^{i,\varepsilon}(0)\), \(\tilde{V}^{i,\varepsilon}_0(t)\)
	has a gaussian distribution with mean 
	\[
	\phi_t(V^{i,\varepsilon}(0))
	\]
	and variance
	\[
	\frac{\varepsilon(1-\exp(-2\gamma t))}{2\gamma}.
	\]
	The proof of \eqref{eq inegalitelemme31} for \(n=0\) is ended by using \eqref{eq comparaisonfnrepart} and an integration with respect to the initial law 
	\(V^{i,\varepsilon}(0)\).
	\[
	\begin{split}
	\mathbb{P}\left(\tilde{V}^{i,\varepsilon}_1(t) + C < \theta\right) =&
	\mathbb{P}\left(\tilde{V}^{i,\varepsilon}_1(t) + C < \theta,\tilde{V}^{i,\varepsilon}_1(\tau_1)\geq \alpha\right)\\
	&+ \mathbb{P}\left(\tilde{V}^{i,\varepsilon}_1(t) + C < \theta,\tilde{V}^{i,\varepsilon}_1(\tau_1) < \alpha\right).
	\end{split}
	\]
	Using conditional probability, an upper bound on the first term is  
	\[
	\Phi\left((\theta-\alpha-C)\sqrt{\dfrac{2\gamma}{\varepsilon}}\right).
	\]
	Furthermore,
	\[
	\begin{split}
	\mathbb{P}\left(\tilde{V}^{i,\varepsilon}_1(t) + C < \theta,\tilde{V}^{i,\varepsilon}_1(\tau_1) < \alpha\right) \leq  &\mathbb{P}\left(\tilde{V}^{i,\varepsilon}_1(\tau_1) < \alpha\right)\\
	\leq  &\mathbb{P}\left(\tilde{V}^{i,\varepsilon}_1(\tau_1) < \alpha,i\in J(1)\right) \\
	&+ \mathbb{P}\left(\tilde{V}^{i,\varepsilon}_1(\tau_1) < \alpha,i\notin J(1)\right)\\
	\leq  &\mathbb{P}\left(\tilde{V}^{i,\varepsilon}_0(\tau_1) + m < \alpha,
	i\notin J(1)\right)\\
	\leq  &\mathbb{P}\left(\tilde{V}^{i,\varepsilon}_0(\tau_1) + m < \alpha\right)\\
	\leq & \Phi\left(-m\sqrt{\frac{2\gamma}{\varepsilon}}\right).
	\end{split}
	\]
	So, for \(n=1\),
	\[
	\mathbb{P}\left(\tilde{V}^{i,\varepsilon}_1(t) + C < \theta\right)
	\leq 
	\Phi\left((\theta-\alpha-C)\sqrt{\dfrac{2\gamma}{\varepsilon}}\right)
	+\Phi\left(-m\sqrt{\frac{2\gamma}{\varepsilon}}\right).
	\]
	To conclude, we use a straightforward induction on \(n\), splitting cases 
	according to the position of \(\tilde{V}^{i,\varepsilon}_n(\tau_n)\)
	with respect to \(\alpha\).
\end{proof}
\begin{lemma}\label{lemmaprobasynchrodechargebcp}
	Let \(A\) be  in \([\frac{\theta - \alpha}{m}, N-1]\).
	\begin{multline}\label{eq:09012015}
	\forall \ell\geq 1,\quad \mathbb{P}\left(\Card(J(\ell))=N	\middle|\Card(J(\ell)) \geq A\right)
	\geq \\
	\left(1  - \Phi\left((\theta-\alpha-mA)\sqrt{\dfrac{2\gamma}{\varepsilon}}\right) - \ell\Phi\left(-m\sqrt{\frac{2\gamma}{\varepsilon}}\right)\right)^N
	\end{multline}
\end{lemma}
\begin{proof}
	Given \(\Card(J(\ell)) \geq A\), a sufficient condition to have 
	\(\Card(J(\ell)) = N\) is 
	\[
	\tilde{V}^{i,\varepsilon}_{\ell-1}(\tau_\ell)+mA \geq \theta, \forall i \in \llbracket 1,N\rrbracket.
	\] 
	On the one hand, assuming \(A \geq \frac{\theta-\alpha}{m}\), we  apply Lemma~\ref{lemmaqueuededistribution} with \(C=Am\). On the other hand,
	using the conditional independence of \((\tilde{V}^{i,\varepsilon}_{\ell-1}(t)-\tilde{V}^{i,\varepsilon}_{\ell-1}(\tau_{\ell-1}))_{t\geq \tau_{\ell-1}}\), 
	given \(\tilde{V}^{i,\varepsilon}_{\ell-1}(\tau_{\ell-1})\), we obtain \eqref{eq:09012015}.
\end{proof}
Since \(N \geq \frac{\theta-\alpha}{m}(\frac{\theta-\alpha}{m}+2)\), there are integers in \([\frac{\theta-\alpha}{m},\frac{mN}{\theta-\alpha}]\).
For 
such an integer \(n\),
we can apply Lemma~\ref{lemmaprobasynchrodechargebcp} with 
\(A = \frac{N}{n}\) and Lemma~\ref{lemma minorantchaussette}
to obtain
\begin{multline*}
\mathbb{P}\left(\sup_{\ell \in\llbracket 1,n\rrbracket} \Card(J(\ell)) = N\right)\geq\\
\left(1 - N\Phi\left(-m\sqrt{\dfrac{2\gamma}{\varepsilon}} \inf\left\{n-\frac{\theta-\alpha}{m} ,\frac{\beta-\theta}{m}\right\}\right) \right)\times\\
\left(1  - \Phi\left(-m\sqrt{\dfrac{2\gamma}{\varepsilon}}\left(\frac{N}{n} - \frac{\theta - \alpha}{m}\right)\right) - n\Phi\left(-m\sqrt{\dfrac{2\gamma}{\varepsilon}}\right)\right)^N.
\end{multline*}

\subsection{Proof of the Result of Stability (Th.~\ref{THSTABSYNCHRO})}\label{PROOFSTABSYNCHRO}
Firstly, we know that the system has the homogeneous Markov property, so 
\[
\left\{V^{i,\varepsilon}(\tau^{\varepsilon}_{n+1})=V_r,\forall\,i\in I 
\right\}\mbox{ conditionned by }
\left\{V^{i,\varepsilon}(\tau^{\varepsilon}_n)=V_r,\forall\,i\in I\right\}
\]
have the same probability for any \(n\). Thus, we prove the result of Theorem~\ref{THSTABSYNCHRO} for \(n=0\), that is we consider
a system of \(N\) neurons initially synchronized, that is
\[
\forall i \in \llbracket 1,N\rrbracket, \quad V^{i,\varepsilon}(0)=V_r.
\]
We estimate in this section
a lower bound for the probability to stay synchronized at the next
spiking time of the network \(\tau^\varepsilon_1\) (see \eqref{eq:definitiontaun} and 
\eqref{eq:evolutiondevtilde}).

The main tool of the proof of  Theorem~\ref{THSTABSYNCHRO}
is the Large Deviation Principle for an Ornstein Uhlenbeck
process. We give an explicit expression of the good rate function in
the following Lemma. 
\begin{lemma}\label{lemme_grandedevornstein}
	Let \(x\)  be the solution of the linear equation
	and \(X^\varepsilon\) the solution of the associated noisy
	dynamic.
	\begin{align}
	X^{\varepsilon}(t) & =  x_{0}+\int_{0}^{t}K-\gamma X^{\varepsilon}(s)ds+\sqrt{\varepsilon}\, W_{t},\\
	x(t) & =  x_{0} +\int_{0}^{t} K - \gamma  x(s)ds.
	\end{align}
	For any fixed level $\delta > 0$ and finite time \(T\), let
	\begin{equation}\label{eq:notationgrandedev}
	P_\varepsilon(\delta,T):=\mathbb{P}\left(\sup_{0\leq s\leq T}\left|X^{\varepsilon}(s)-x(s)\right|\geq\delta\right).
	\end{equation}
	We have
	\[
	\lim_{\varepsilon\rightarrow0}\varepsilon\log P_\varepsilon(\delta,T)
	=-\frac{\gamma\delta^{2}}{2}\left(1+\frac{\cosh\left(\gamma T\right)}{\sinh\left(\gamma T\right)}\right) < -\gamma\delta^{2}.
	\]
\end{lemma}
\begin{proof}
	See Appendix.
\end{proof}

Now we finish the proof of the Theorem~\ref{THSTABSYNCHRO}. A suffisant condition for the network to stay synchronized at 
\(\tau^\varepsilon_1\) is
\begin{equation}\label{eq:onrestesynchronise}
\tilde{V}^{i,\varepsilon}_0(\tau^{\varepsilon}_1)+m\geq\theta\quad\forall\,i\in I.
\end{equation}
Let us denote 
\[
v(t)=\phi_t(V_r) = (V_r-\beta)\exp(-\gamma t)+\beta
\] 
the solution to the associated deterministic equation (i.e. \(\varepsilon = 0\))
starting from the reset potential $V_r$. 
We fix \(\delta\in(0,\min\{\beta-\theta,\theta-V_r\})\),
then there exist $0<t_1(\delta)<t_2(\delta)$ such that $v(t_1(\delta))+\delta=\theta$ and $v(t_2(\delta))-\delta=\theta$.   

Assume
\begin{equation}\label{TUBE}
v(t)-\delta< \tilde{V}^{i,\varepsilon}_0(t)< v(t)+\delta\quad\forall\,t\in[0,t_2(\delta)] \quad\forall\,i\in I,
\end{equation}
then $\tau^{\varepsilon}_1\in(t_1(\delta),t_2(\delta))$. Since $v(t_1(\delta))<v(\tau^{\varepsilon}_1)$, it follows that
\[
v(t_1(\delta))-\delta< \tilde{V}^{i,\varepsilon}_0(\tau^{\varepsilon}_1) \quad \forall\,i\in I,
\]
i.e.
\[
\tilde{V}^{i,\varepsilon}_0(\tau^{\varepsilon}_1)+m>\theta-2\delta + m \quad \forall\,i\in I.
\] 
Therefore, if \eqref{TUBE} is satisfied for \(\delta\leq m/2\), then 
all the neurons spike at time \(\tau^\varepsilon_1\) and so the network stays synchronized.
As the potential \(\tilde{V}^{i,\varepsilon}_0\)
are independent before \(\tau^\varepsilon_1\),
condition \eqref{TUBE} is fulfilled with a probability greater than
\[
(1-P_{\varepsilon}(\delta,t_2(\delta)))^N
\qquad\mbox{ where }\quad t_2(\delta)=\frac{1}{\gamma}\log\left(\frac{\beta}{\beta-\theta-\delta}\right),
\]
and the probability of synchronization satisfies
\[
\mathbb{P}(V^{i,\varepsilon}(\tau^{\varepsilon}_1)=0,\forall\,i\in I)\geq\left(1-P_{\varepsilon}\left(\frac{m}{2},t_2\left(\frac{m}{2}\right)\right)\right)^N.
\]
From Lemma~\ref{lemme_grandedevornstein} it follows that 
\[ 
\mathbb{P}(V^{i,\varepsilon}(\tau^{\varepsilon}_1)=0,\forall\,i\in I)\geq
\left(1-\exp\left(-\gamma\frac{m^2}{4\varepsilon}\right)\right)^N
\]
for all $\varepsilon$ small enough. 

\appendix
\section{Large Deviation Principle for
	Ornstein Uhlenbeck Process}
We  recall a result of 
Freidlin-Wentzell theory and apply it
to the particular case of Ornstein-Uhlenbeck process.

We consider a one dimensional deterministic function \((x(t),t\geq 0)\)
evolving according to
\begin{equation}
x(t) = x_{0}+\int_{0}^{t}b(x(s))ds,
\end{equation}
where \(b: \mathbb{R}\mapsto\mathbb{R}\) is a uniformly Lipschitz continuous function.  
The main question is related to the \textit{distance}
between \(x\) and its stochastic perturbation by 
an additive noise:
\begin{equation}
X^{\varepsilon}(t)  =  x_{0}+\int_{0}^{t}b(X^{\varepsilon}(s))ds+\sqrt{\varepsilon}\, W_{t}.
\end{equation}
Precisely, we aim to estimate
\[
\mathbb{P}\left(\sup_{0\leq t\leq T}\left|X^{\varepsilon}(t)-x(t)\right|\geq\delta\right),
\]
for \(T>0\) and \(\delta>0\) (fixed).
Answers to this classical question 
can be obtained thanks to Freidlin-Wentzell theory.
We introduce the functional set 
\[
H_1^{x_0}:=\left\{f:f(t)=x_0+\int_0^t\phi(s)ds,\phi\in L_2([0,T]), t\in[0,T]\right\}.
\]
\begin{theorem}
	The family \((X^\varepsilon(t),0\leq t\leq T)\)
	satisfies a large deviation principle 
	with good rate function \(I_{x_0}\)
	\begin{equation*}
	I_{x_0}(f)  =  
	\begin{cases}
	\dfrac{1}{2}\displaystyle\int_{0}^{T}\left|f^{\prime}(t)-b(f(t))\right|^{2}dt\quad\mbox{ for }f\in H_{1}^{x_0}\\
	+\infty\qquad\mbox{ otherwise.}
	\end{cases}
	\end{equation*}
\end{theorem}
The proof can be found e.g. in \cite[Sect. 5.6]{dembozeitouni}. We apply with the set 
\[
\Gamma=\left\{ f\in C_{0}[0,T],\,\sup_{0\leq t\leq T}\left|f(t)-x(t)\right|\geq\delta\right\}.
\]
The 
LDP writes
\begin{multline*}
-\inf_{f\in\overset{\circ}{\Gamma}}I_{x_0}(f)  \leq  \liminf_{\varepsilon}\varepsilon\log\mathbb{P}\left(\sup_{0\leq t\leq T}\left|X^{\varepsilon}(t)-x(t)\right|\geq\delta\right)\\
\limsup_{\varepsilon}\varepsilon\log\mathbb{P}\left(\sup_{0\leq t\leq T}\left|X^{\varepsilon}(t)-x(t)\right|\geq\delta\right) \leq -\inf_{f\in\bar{\Gamma}}I_{x_0}(f),
\end{multline*}
where \(\overset{\circ}{\Gamma}\) and \(\bar{\Gamma}\) denote   the interior and the closure of \(\Gamma\) respectively.
In our case the first and last quantities are equal, so it remains to evaluate
\[
\inf_{f\in\Gamma}I_{x_0}(f)
=\inf_{f\in\Gamma\cap H_1^{x_0}}\frac{1}{2}\int_{0}^{T}\left|f^{\prime}(t)-b(f(t))\right|^{2}dt.
\]
As $b(x)=-\gamma x + K$, we  find an explicit
expression for the previous quantity.

Denote $\varphi(t)=f(t)-x(t)$, we have
\[
\int_{0}^{T}\left|f^{\prime}(t)-b(f(t))\right|^{2}dt  = 
\int_{0}^{T}\left|\varphi^{\prime}(t)+\gamma \varphi(t))\right|^{2}dt.
\]

So, it remains to evaluate:
\[
\inf_{
	\varphi\in H_{1}^{0},\|\varphi\|_{\infty}\geq\delta
}
\int_{0}^{T}\left|
\varphi^{\prime}(t)+\gamma \varphi(t)\right|^{2}dt.
\]

Consider \(
\varphi\in H_{1}^{0},\|\varphi\|_{\infty}\geq\delta
\)
and introduce 
$T_{1}=\inf\left\{ t\in[0,T],|\varphi(t)|=\delta\right\}$. 
We define a modification \(\tilde{\varphi}\in H_{1}^{0}\) such that  
\(|\tilde{\varphi}(T_1)|=\delta\):
\[
\tilde{\varphi}(t) =
\begin{cases}
\varphi(t) \mbox{ for } t\in[0,T_1]\\
\delta \exp\left(-\gamma (t-T_1)\right) \mbox{ for } t\in [T_1,T].
\end{cases}
\]
We have
\[
\int_{0}^{T}\left|\varphi^{\prime}(t)+\gamma\varphi(t))\right|^{2}dt \geq \int_{0}^{T}\left|\tilde{\varphi}^{\prime}(t)+\gamma\tilde{\varphi}(t)\right|^{2}dt =
\int_{0}^{T_1}\left|\tilde{\varphi}^{\prime}(t)+\gamma\tilde{\varphi}(t)\right|^{2}dt.
\]
So, we have to evaluate
\[
\inf_{0\leq T_{1}\leq T}\inf_{\varphi\in H_{1}^{0},|\varphi(T_{1})|=\delta}\int_{0}^{T_{1}}\left|\varphi^{\prime}(t)+\gamma\varphi(t)\right|^{2}dt.
\]
We have
\begin{eqnarray*}
	\int_{0}^{T_{1}}\left|\varphi^{\prime}(t)+\gamma\varphi(t)\right|^{2}dt & = & \int_{0}^{T_{1}}\left(\varphi^{\prime}\right)^{2}(t)+\gamma^{2}\varphi^{2}(t)dt+\gamma\int_{0}^{T_{1}}2\varphi(t)\varphi^{\prime}(t)dt\\
	& = & \int_{0}^{T_{1}}\left(\varphi^{\prime}\right)^{2}(t)+\gamma^{2}\varphi^{2}(t)dt+\gamma\delta^{2}.
\end{eqnarray*}
Let 
\[
\Psi(\varphi)=\int_{0}^{T_{1}}\left(\varphi^{\prime}(t)\right)^{2}+\gamma^{2}\varphi^{2}(t)dt \quad\mbox{ for }\quad\varphi\in H_{1}^{0},|\varphi(T_{1})|=\delta.
\]
The infimum of \(\Psi\) is reached by 
\[
\varphi_{T_1}(t)=\delta\frac{\sinh\gamma t}{\sinh\gamma T_{1}} \quad \mbox{ for } t\in[0,T_1].
\]
since for all \(\Lambda\in H_1^0\) such that \(\Lambda(T_1)=0\), 
\begin{eqnarray*}
	\Psi(\varphi_{T_1} + \Lambda) & = & \int_{0}^{T_{1}}\left(\varphi_{T_1}^{\prime}(t)+\Lambda^{\prime}(t)\right)^{2}+\gamma^{2}\left(\varphi_{T_1}(t)+\Lambda(t)\right)^{2}dt\\
	& = & \Psi(\varphi_{T_1})+\Psi(\Lambda)+2\int_{0}^{T_{1}}\varphi_{T_1}^{\prime}(t)\Lambda^{\prime}(t)dt+2\gamma^{2}\int_{0}^{T_{1}}\varphi_{T_1}(t)\Lambda(t)dt\\
	& = & \Psi(\varphi_{T_1})+\Psi(\Lambda)+2\int_{0}^{T_{1}}\Lambda(t)(-\varphi_{T_1}^{\prime\prime}(t)+\gamma^{2}\varphi_{T_1}(t))dt\\
	& = & \Psi(\varphi_{T_1})+\Psi(\Lambda) \geq \Psi(\varphi_{T_1}).
\end{eqnarray*}
\begin{eqnarray*}
	\Psi(\varphi_{T_1}) & = & \frac{\delta^{2}\gamma^{2}}{\sinh^{2}\left(\gamma T_{1}\right)}\int_{0}^{T_{1}}\cosh^{2}\left(\gamma t\right)+\sinh^{2}\left(\gamma t\right)dt\\
	& = & \frac{\delta^{2}\gamma^{2}}{\sinh^{2}\left(\gamma T_{1}\right)}\int_{0}^{T_{1}}\cosh\left(2\gamma t\right)dt
	=  \delta^{2}\gamma\frac{\cosh\left(\gamma T_{1}\right)}{\sinh\left(\gamma T_{1}\right)}.
\end{eqnarray*}
Therefore, the function \(\Psi(\varphi_{T_1})\) is
a decreasing function of \(T_1\) and 
\begin{eqnarray*}
	\inf_{f\in\bar{\Gamma}}I(f) & = & \frac{1}{2}\left(\Psi\left(\varphi_{T}\right)+\gamma\delta^{2}\right)\\
	& = & \frac{\gamma\delta^{2}}{2}\left(1+\frac{\cosh\left(\gamma T\right)}{\sinh\left(\gamma T\right)}\right).
\end{eqnarray*}

\[
\lim_{\varepsilon\rightarrow0}\varepsilon\log\mathbb{P}\left(\sup_{0\leq t\leq T}\left|X^{\varepsilon}(t)-x(t)\right|\geq\delta\right)=-\frac{\gamma\delta^{2}}{2}\left(1+\frac{\cosh\left(\gamma T\right)}{\sinh\left(\gamma T\right)}\right)<-\gamma\delta^{2}.
\]

\section*{Acknowledgments}  
PG and ET have been supported by  Mathamsud Project 13MATH-04 - SIN. PG thanks University of Valparaiso project DIUV 45/2013.
ET has also been supported by the  Inria Equipe Associ\'ee Anestoc-Tosca
and by the European Union's Horizon 2020 Framework Programme for Research and Innovation under Grant Agreement No. 720270 (Human Brain Project SGA1).
ET is
grateful to Inria Sophia Antipolis - M\'editerran\'ee
``Nef'' computation cluster for providing resources and support.



\end{document}